\documentclass[reqno]{amsart}
\usepackage{tkz-euclide}
\usepackage{hyperref}
\usepackage{amsmath}
\usepackage{amssymb}
\usepackage{mathtools}
\usepackage{graphicx}
\usepackage{tikz}
\usetikzlibrary{calc}
\usepackage{xcolor}
\usetikzlibrary{snakes}
\usepackage{amsthm}
\usepackage{float}
\usepackage[utf8]{inputenc}
\usetikzlibrary{arrows.meta,snakes} 
\newcommand{\EEE}{\color{black}} 

\newcommand{\eps}{\varepsilon} 

\theoremstyle{plain}
\begingroup
\newtheorem{theorem}{Theorem}[section]
\newtheorem{lemma}[theorem]{Lemma}
\newtheorem{proposition}[theorem]{Proposition}

\endgroup

\numberwithin{equation}{section}
\newcommand{\R}{\mathbb{R}}

\theoremstyle{definition}
\begingroup

\newtheorem{remark}[theorem]{Remark}

\endgroup

\usepackage[english]{babel}

\begin{document}

\title[A note on the Winterbottom shape]{A note on the Winterbottom shape}


\author{Leonard Kreutz}
\address[Leonard Kreutz]{School of Computation, Information and Technology, Technical University of Munich\\
Boltzmannstra\ss e 3, 85748 Garching bei M\"unchen, Germany.}
\email{leonard.kreutz@tum.de}

\author{Bernd Schmidt}
\address[Bernd Schmidt]{Institut f{\"u}r Mathematik, Universit{\"a}t Augsburg, 
Universit{\"a}tsstr.\ 14, 86159 Augsburg, Germany.}
\email{bernd.schmidt@math.uni-augsburg.de}

\date{\today}  

\maketitle

\begin{abstract} In this short note we review results on equilibrium shapes of minimizers to the sessile drop problem. More precisely, we study the Winterbottom problem and prove that the Winterbottom shape is indeed optimal. The arguments presented here are based on relaxation and the (anisotropic) isoperimetric inequality.
\end{abstract}

\section{Introduction}

For solid crystals with sufficiently small grains, Herring \cite{Herring} claims that the bulk contribution to their configurational energy is negligible with respect to the surface tension. Thus, in this setting, when determining the equilibrium shape of a crystal, interfacial energies of the type
\begin{align*}
P_{\varphi}(E) = \int_{\partial^*E}\varphi(\nu_E)\,\mathrm{d}\mathcal{H}^{d-1}\,
\end{align*}
play a fundamental role. Here, $\varphi \colon \mathbb{R}^d\to [0,+\infty)$ is a convex and positively homogeneous function of degree one that describes the possibly anisotropic surface tension of the crystal. The set $E$ is assumed to be sufficiently regular (i.e.~$E$ is a set of finite perimeter), $\partial^*E$ denotes its (measure theoretic) boundary, and $\nu_E(x)$ the outer normal to the set $E$ at the point $x \in \partial^*E$. The Wulff problem consists in studying solutions of
\begin{align}\label{intropb:Wulff}
E \in \mathrm{argmin} \left\{P_{\varphi}(E) \colon |E|=v\right\}\,,
\end{align}
where $v>0$ is given. This problem is an anistropic generalization of the classical isoperimetric problem. In the early 1900s Wulff \cite{Wulff} proposed a geometric construction to \eqref{intropb:Wulff} given by
\begin{align}\label{introeq:Wulffset}
W_\varphi = \{x \in \mathbb{R}^d \colon x\cdot\nu \leq  \varphi(\nu) \text{ for all } \nu \in \mathbb{S}^{d-1}\}\,.
\end{align}
This shape is nowadays known as the Wulff set (or Wulff crystal) of $\varphi$. Dinghas proved formally in \cite{Dinghas} that among convex polyhedra the Wulff set is the shape having the least surface integral for the volume it contains. The proof has later been rendered precise by Taylor \cite{Taylor74,Taylor75,Taylor} using arguments from geometric measure theory.

The Wulff variational problem provides a description of an equilibrium crystal shape deep inside a region in the gas phase. This leads to the natural question if, likewise, the shape of a crystal growing on a substrate can be determined that minimizes a suitable combination of its surface tension and the interaction energy with the substrate. Such a situation may be described as follows. For $\varphi$ be as above and $\lambda \in \mathbb{R}$ we set
\begin{align}\label{intro-eq:energy}
F_{\lambda,\varphi}(E) = \int_{\partial^* E\cap H^+} \varphi(\nu_E)\,\mathrm{d}\mathcal{H}^{d-1} + \lambda\mathcal{H}^{d-1}(\partial^* E\cap H)\,,
\end{align}
where $H^+=\{x \in \mathbb{R}^d \colon x_d >0\}$, $H=\{x \in \mathbb{R}^d \colon x_d=0\}$ and $E$ is a set of finite perimeter in $H^+$, i.e.\ $|E \setminus H^+| = 0$, which we simply denote by $E \subseteq H^+$,  cf.~Fig.~\ref{fig:configuration}. 
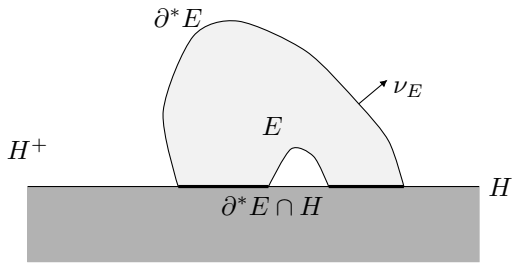
\begin{figure}[htp]
\begin{tikzpicture}
\tikzset{>={Latex[width=1mm,length=1mm]}};

\draw [fill=gray!10!white] plot [smooth] coordinates {(-1,0)(-1.2,1)(-.8,2)(-.2,2.2)(.7,1.8)(1.4,1.1)(1.8,.6)(2,0)};

\draw[fill=white] plot [smooth] coordinates {(.2,0)(.5,.5)(.8,.4)(1,0)};

\draw(-1,2) node [anchor=south]{$\partial^* E$};

\draw(.25,.8) node {$ E$};

\draw[->](1.4,1.1)--++(40:.5);

\draw(1.35,1.2)++(40:.5)node [anchor=north west]{$\nu_E$};

\draw [fill=gray!60!white,gray!60!white](-3,0)--++(6,0)--++(0,-1)--++(-6,0)--++(0,1);

\draw[very thick](-1,0)--(.2,0);
\draw[very thick](1,0)--(2,0);

\draw (-3,0)--++(6,0);

\draw(3,0) node[anchor=west]{$H$};

\draw(-3,.5) node{$H^+$};

\draw(.25,-.5) node[anchor=south] {$ \partial^* E \cap H$};

\end{tikzpicture}
\caption{A generic configuration admissible to \eqref{intropb:Winterbottom}. The substrate is illustrated in dark gray, the region occupied by the crystal is illustrated in light gray, and the contact surface is illustrated in bold.}
\label{fig:configuration}
\end{figure}
Here, as in the Wulff problem, $\varphi$ represents the (anisotropic) surface tension density between crystal and vapor and $\lambda \in \mathbb{R}$ is the relative adhesion coefficient between the crystal and the substrate (i.e., the difference of the crystal-substrate and the substrate-vapor interfacial energy per unit surface area). In this setting, the Winterbottom problem \cite{Winterbottom} consists in studying the existence and finding  the solutions of 
\begin{align}\label{intropb:Winterbottom}
E \in \mathrm{argmin}\left\{F_{\lambda,\varphi}(E) \colon E \subseteq H^+, |E| =v \right\}\,,
\end{align}
where $v>0$. There are three interesting parameter regimes to consider (see Remark~\ref{rmk:drying-wetting}): 
\begin{itemize}
\item[(1)] If $\lambda \geq \varphi(-e_d)$ it is energetically inconvenient for the crystal to attach to the substrate and therefore the solution \eqref{intropb:Winterbottom} coincides (up to translation) with the solution of \eqref{intropb:Wulff}. This phenomenon is called complete drying.
\item[(2)] If $\lambda \in (-\varphi(e_d),\varphi(-e_d))$ a solution to the Winterbottom exists, but its shape differs from the Wulff shape given in \eqref{introeq:Wulffset}. In fact, the solution shape is now the Winterbottom shape given by
\begin{align}\label{introeq:WinterbottomShape}
W_{\lambda,\varphi} = W_\varphi \cap \{x\in \mathbb{R}^d \colon x_d \geq -\lambda\}
\end{align}
suitably rescaled and placed in order to be in contact with the substrate. In this case we speak of partial drying/wetting.
\item[(3)] If $\lambda \leq -\varphi(e_d)$ complete wetting occurs: It is energetically favorable to create large surface area in common with the substrate. This allows to create arbitrarily small energy (if $\lambda= -\varphi(e_d)$) or even energy diverging to $-\infty$ (if $\lambda< -\varphi(e_d)$).
\end{itemize}
We rigorously study all the cases. From the point of view of the mathematical analysis, the most interesting regime to consider is the partial wetting regime (2). There, our first main result Theorem \ref{thm:main} proves that indeed \eqref{introeq:WinterbottomShape} solves \eqref{intropb:Winterbottom}.

In the case that $\varphi$ is smooth the solution given in \eqref{introeq:WinterbottomShape} in particular recovers Young's law \cite{Young}. This law relates the contact angle of the boundary of equilibrium configurations with the adhesion coefficient $\lambda$. More precisely, if $\partial(\overline{E}\cap H)$ denotes the boundary of $\overline{E}\cap H$ in $H$, then  
\begin{align}\label{introeq:Young}
\nabla \varphi(\nu_E(x)) \cdot (-e_d) = \lambda\,, \quad \text{ for all } x \in \partial(\overline{E}\cap H)\,
\end{align}
 (see Remark~\ref{rmk:drying-wetting}\eqref{proof-YL} below).  In the isotropic case, i.e.~$\varphi(\nu) = |\nu|$, this reads (cf.~Fig.~\ref{fig:Young's law})
\begin{align*}
\nu_E(x) \cdot (-e_d) = \lambda \,, \quad \text{ for all } x \in \partial(\overline{E}\cap H)\,.
\end{align*}
The equilibrium condition \eqref{introeq:Young} has been derived  more generally  for anistropic capillarity problems in \cite{DePhilippisMaggi} and for epitaxially-strained thin films in \cite{DavoliPiovano}. 
\begin{figure}[htp]
\begin{tikzpicture}

\clip(4,-1)--++(-4,0)--++(0,2)--++(4,0)--++(0,-2);
\tikzset{>={Latex[width=1mm,length=1mm]}};

\draw [fill=gray!10!white](2,-1) arc (0:180:2);

\draw[->](1.725,0)--++(40:.5);

\draw(1.725,.2)++(40:.5)node [anchor=north west]{$\nu_E$};

\draw [fill=gray!60!white,gray!60!white](-3,0)--++(7,0)--++(0,-1)--++(-7,0)--++(0,1);

\draw[very thick](-1.725,0)--(1.725,0);

\draw (-3,0)--(1.725,0);

%

\draw(1.725,0)++(270:.5)node [anchor= east]{$-e_d$};

\draw[->](1.725,0)--++(270:.5);

\draw[fill=black](1.725,0) circle(.05);

\draw(1.725,0) node[anchor=north west]{$x \in \partial(\overline{E}\cap H)$};

\end{tikzpicture}
\caption{Young's law for the contact angle. $\nu_E$ is the normal at the point $x \in \partial(\overline{E}\cap H)$ and $-e_d$ is the normal vector of $H$ pointing outwards with respect to the region the crystal may occupy.}
\label{fig:Young's law}
\end{figure}
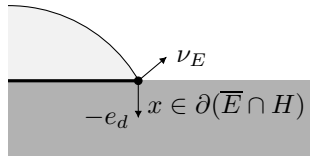

The proof that \eqref{introeq:WinterbottomShape} is the solution shape to \eqref{intropb:Winterbottom} provides a rigorous justification of an ansatz that has been previously considered in \cite{KoteckyPfister}. It essentially lies on the anistropic isoperimetric inequality \cite{EvansGariepy92} for the case $\lambda >0$ and change of coordinates in the case that $\lambda \leq 0$. In particular the quantitative isoperimetric inequality \cite{CicaleseLeonardi,FuscoMaggiPratelli,FigalliMaggiPratelli}, implies the stability of the Winterbottom shape. The stability of the Winterbottom shape has already been proved in two dimensions in \cite{KoteckyPfister} with the use of the (generalized) Bonnesen inequalities \cite{Bonnesen}. The stability of the Winterbottom shape in any dimension is the content of our second main Theorem~\ref{thm:main-quant}: For a set $E \subset H^+$ we show that its squared asymmetry index (i.e., the Lebesgue measure of the symmetric difference of $E$ with the optimally placed Winterbottom shape of equal volume) is controlled by the isoperimetric difference (the energy deficit of $E$ and the Winterbottom shape).

The paper is organized as follows. In Section \ref{sec:settings}, we introduce the problems, discuss the different parameter regimes, and state the main theorems. In Section \ref{sec:proofs}, we prove the main theorem. Here, we would like to point out that, due to the presence of substrate, we are working with discontinuous integrands in general and thus, some of the results already present for continuous integrands need to be proved again in our setting, cf.~Proposition~\ref{prop:basic-Wulff}.

\section{Settings and Main results}\label{sec:settings}

\subsection*{Notation and Preliminaries}
 
 For a measurable subset $B$ of $\mathbb{R}^d$ we denote by $|B|$ its $d$-dimensional Lebesgue measure and by $\mathcal{H}^k(B)$ its $k$-dimensional Hausdorff measure. Given $x,y \in \mathbb{R}^d$ we denote by $x \cdot y$ their scalar product and by $\|x\|$ the Euclidean norm of $x$. For $r > 0$, $x \in \mathbb{R}^d$ we write $B_r(x) = \{y \in \mathbb{R}^d \colon \|y-x\| < r\}$ and abbreviate $\mathbb{S}^{d-1} = \partial B_1(0)$. We also introduce the half space $H^+ =\{x \in \mathbb{R}^d \colon  x_d >0\}$ and the hyperplane $H=\{x \in \mathbb{R}^d \colon x_d=0\}$. 
 If $E \subset \mathbb{R}^d$ is a set of finite perimeter (i.e., its characteristic function is of bounded variation), we write $\partial^* E$ for the reduced boundary of $E$ and denote by $\nu_E \colon \partial^* E \to \mathbb{S}^{d-1}$ the generalized outer normal to $E$. We refer to \cite{AFP} for the definition of these objects and basic facts about sets of finite perimeter. In particular, if $K \subseteq  \mathbb{R}^d$ is a bounded convex set, then $K$ is of finite perimeter and $\nu_K(x) \in \mathcal{N}_K(x)$ for all $x \in \partial^* K \subseteq \partial K$, where, for each $x \in \partial K$, $\mathcal{N}_K(x):=\{\zeta \in \R^d \colon \zeta \cdot (x'-x) \le 0 ~\forall\, x' \in K\}$ denotes the normal cone to $K$ at $x$.

\subsection*{Energy}  Let $\varphi \colon \mathbb{R}^d\to [0,+\infty)$ be a convex and positively homogeneous function of degree one that is bounded from below, i.e. there exists $c>0$ such that
\begin{align}\label{ineq:bounded below}
\varphi(\nu) \geq c\, \|\nu\| \text{ for all } \nu \in \mathbb{R}^d\,.
\end{align}
For a set of finite perimeter $E \subseteq H^+$ we define $F_{\lambda,\varphi}(E)$ as in \eqref{intro-eq:energy}. Given $v>0$, we are interested in studying the shape of the solutions to \eqref{intropb:Winterbottom}. 
\begin{remark}[Scaling]\label{rem:scaling} It is obvious that $F_{\lambda,\varphi}(r E) = r^{d-1} F_{\lambda,\varphi}(E)$ for any $r > 0$ and, in particular, the minimal energy $m_{\lambda,\varphi}(v) = \inf\left\{F_{\lambda,\varphi}(E) \colon |E| = v \right\}$ satisfies    
\begin{align*}
m_{\lambda,\varphi}(v) = v^{\frac{d-1}{d}} m_{\lambda,\varphi}(1)\,.
\end{align*}
\end{remark}

\subsection*{Wulff shape} In order to construct solutions to \eqref{intropb:Winterbottom} we first define the \emph{Wulff set} of $\varphi$, yet in a more general set-up: For any positively $1$-homogeneous Borel function $\varphi \colon \mathbb{R}^d \to [0,+\infty)$ which is bounded from below (cf.\ \eqref{ineq:bounded below}) but not necessarily continuous or even convex, we set 
\begin{align*}
W_\varphi 
:= \{ x \in \mathbb{R}^d\colon\nu \cdot x \leq \varphi(\nu) ~\forall\, \nu \in \mathbb{S}^{n-1}\}\,.
\end{align*}

In fact, for continuous $\varphi$, $W_\varphi $ is the -- up to translations unique -- minimizer among sets with equal volume to the problem $E \mapsto \int_{\partial^* E} \varphi(\nu_E)\,\mathrm{d}\mathcal{H}^{d-1}$ without substrate, cf.\ \cite{Taylor,Fonseca,FonsecaMueller}. 
We recall some basic facts on Wulff sets. 
\begin{proposition}\label{prop:basic-Wulff}  \
\begin{enumerate}
 \item\label{eq-phi-elementary} $W_\varphi$ is a bounded, convex and closed set with $0 \in \operatorname{int} W_\varphi$. 
 \item\label{eq-phi-conj} The convex conjugate $\varphi^*$ is equal to the indicator function of $W_\varphi$, i.e., $\varphi^*(x) = 0$ if $x \in W_\varphi$ and $\varphi^*(x) = \infty$ if $x \notin W_\varphi$. 
 \item\label{char-normal-cone} For $x \in \partial W_\varphi$ we have that $\mathcal{N}_{W_\varphi}(x)=\partial \varphi^*(x)$, where $\partial \varphi^*(x)$  is the subgradient of $\varphi^*$ at $x$. 
 \item\label{W-conv-env} The Wulff shapes of $\varphi$ and its convex envelope $\varphi^{**}$ coincide: $W_\varphi = W_{\varphi^{**}}$. 
 \item\label{eq-normal-cone} If $x \in \partial W_\varphi$ and $\zeta$ lies in the the normal cone to $W_\varphi$ at $x$, then $x \cdot \zeta = \varphi^{**}(\zeta)$. 
\end{enumerate}
\end{proposition}

\begin{proof} Convexity and closedness of $W_\varphi$ and $0 \in B_c \subseteq W_\varphi$ for $c>0$ given by \eqref{ineq:bounded below} are immediate. Boundedness follows in our set-up from $-\varphi(-e_k) \leq x_k \leq \varphi(e_k)$ for $k=1,\ldots,n$ for any $x \in W_\varphi$. This proves \eqref{eq-phi-elementary}. 
Now, \eqref{eq-phi-conj}--\eqref{eq-normal-cone} \EEE are found in \cite[Propositions 3.4 \& 3.5]{Fonseca}. A careful insepection of the proofs in \cite{Fonseca} reveals that the relevant assertions do not need continuity of $\varphi$ (which is assumed there). The statement in \eqref{eq-normal-cone} follows since, by \eqref{char-normal-cone}, $\zeta \in \partial \varphi^*(x)$, whence the Fenchel-Young identity implies $x \cdot \zeta = \varphi^*(x) + \varphi^{**}(\zeta)$ and so   $x \cdot \zeta = \varphi^{**}(\zeta)$ by \eqref{eq-phi-conj}. 
\end{proof}

\subsection*{Winterbottom shape} 
We now return to our convex, positively $1$-homogeneous and lower bounded $\varphi$ and set 
\begin{align*}
W_{\lambda,\varphi} = W_\varphi \cap \{x\in \mathbb{R}^d \colon x_d \geq -\lambda\}\,.
\end{align*}
 If $|W_{\lambda,\varphi}| > 0$, the \emph{Winterbottom set} with volume $v>0$ is then defined as
\begin{align}\label{def:Winterbottomshape-volume}
W_{\lambda,\varphi}(v) = \left(\frac{v}{|W_{\lambda,\varphi}|}\right)^{\frac{1}{d}} \left(W_{\lambda,\varphi} +\lambda e_d\right)\,.
\end{align}

\subsection*{Minimizers} We are describing quickly the different regimes for $\lambda$:
\begin{enumerate}
\item \emph{Complete drying:}  $\lambda \geq \varphi(-e_d)$. In this case $W_{\lambda,\varphi} = W_{\varphi}$. In Remark~\ref{rmk:drying-wetting} below we will see that minimizers to $F_{\lambda,\varphi}$ for a given volume are those of the unconstrained system without substrate. 

\item \emph{Partial drying/wetting:} $ \lambda \in \left(- \varphi(e_d), \varphi(-e_d)\right)$. 
 Here we have  $W_{\lambda,\varphi} \subsetneq W_{\varphi}$ as shown in  Lemma~\ref{lem:auxiliary} below. Minimality of the Winterbottom shape will be established in Theorem~\ref{thm:main} below.

\item \emph{Complete wetting:} $\lambda \leq  - \varphi(e_d)$.  Here $W_{\lambda,\varphi} = \emptyset$. In Remark~\ref{rmk:drying-wetting} below we will see that solutions to \eqref{intropb:Winterbottom} do not exist if $\lambda < - \varphi(e_d)$ as ``wetting'' configurations that intersect large parts of $H$ have arbitrarily small energy. Generically there are no minimizers in the special case $\lambda = - \varphi(e_d)$ either, while here for particular $\varphi$ (even non-unique) minimizers might exist.   
\end{enumerate}

The first main theorem is the following characterization of minimizers in the partial drying/wetting regime.

\begin{theorem}\label{thm:main} Let $ \lambda \in \left(- \varphi(e_d), \varphi(-e_d)\right)$ and $v >0$. Then
\begin{align*}
F_{\lambda,\varphi} (W_{\lambda,\varphi}(v)) \leq F_{\lambda,\varphi}(E)
\end{align*}
for all $E  \subseteq H^+$ sets of finite perimeter such that $|E|=v$. Moreover, equality holds if and only if there exists $\tau \in H$ such that $|E\Delta \left(W_{\lambda,\varphi}(v) +\tau\right)| =0$.
\end{theorem}

\begin{remark}\label{rmk:drying-wetting} 
\begin{enumerate}
\item \emph{Complete drying:}  $\lambda \geq \varphi(-e_d)$.  Comparison to the unconstrained case shows that $E$ is a solution to \eqref{intropb:Winterbottom} if and only if  $|E\Delta \left(W_{\lambda,\varphi}(v) +\tau\right)| =0$ (recall definition \eqref{def:Winterbottomshape-volume} of $W_{\lambda,\varphi}(v)$) for some $\tau \in H^+$ (if $\lambda > \varphi(-e_d)$ and $\mathcal{H}^{d-1}(\{x \in W_\varphi : x_d = - \varphi(-e_d) \}) > 0$), respectively, $\tau \in H^+\cup H$ (if $\lambda = \varphi(-e_d)$ or $\mathcal{H}^{d-1}(\{x \in W_\varphi : x_d = - \varphi(-e_d) \}) = 0$).

\item\label{positive-m} \emph{Complete drying and partial drying/partial wetting:} $\lambda >  - \varphi(e_d)$.  For every set of finite perimeter $E \subseteq H^+$ one has 
\begin{align}\label{eq:Gauss-Green} 
   \int_{\partial^* E \cap H^+} \varphi(\nu_E)\,\mathrm{d}\mathcal{H}^{d-1} 
   \ge \varphi \Big( \int_{\partial^* E \cap H^+} \nu_E\,\mathrm{d}\mathcal{H}^{d-1} \Big) 
   =  \mathcal{H}^{d-1}(\partial^* E \cap H) \varphi(e_d)\,, 
\end{align}
where we have used Jensen's inequality, the homogeneity of $\varphi$ and the fact that, by the Gauss-Green Theorem for sets of finite perimeter, $\int_{\partial^* E} \nu_E\,\mathrm{d}\mathcal{H}^{d-1} = 0$. This shows that $m_{\lambda,\varphi}$ (cf.\ Remark~\ref{rem:scaling}) is positive if $\lambda >  - \varphi(e_d)$.  
\EEE 

\item \emph{Complete wetting:} $\lambda \leq  - \varphi(e_d)$.  In case $\lambda <  - \varphi(e_d)$ one may consider cylindrical sets $E_R = (0,R)^{d-1} \times (0,v/R^{d-1})$ with $R \to \infty$ to see that $m_{\lambda,\varphi} = - \infty$. In particular, solutions to \eqref{intropb:Winterbottom} do not exist. In case $\lambda = - \varphi(e_d)$ the trial configurations $E_R$ show that $m_{\lambda,\varphi} \le 0$. Together with \eqref{eq:Gauss-Green} this implies $m_{\lambda,\varphi} = 0$. However, in this case both existence and non-existence of minimizers might occur: If, e.g., $\varphi$ is strictly convex, the above argument shows that $F_{\lambda,\varphi}(E) > 0$ for $|E| > 0$. If, by way of contrast, $\varphi$ is affine near $e_d$ the above computation shows that for a spherical cap $C_\varepsilon = \{x \in B_1(0) : x_d \ge 1-\varepsilon\}$, $0 < \varepsilon \ll 1$, the set $E = \frac{v}{|C_\varepsilon|}(C_\varepsilon - (1 - \varepsilon) e_d)$ has $F_{\lambda,\varphi}(E) = 0$. 

\item\label{proof-YL} \emph{Young's law:} The relation \eqref{introeq:Young} can be seen as follows. Let $x \in \partial W_{\lambda,\varphi}$ with $x_d = - \lambda$ and $\nu(x) \in \mathcal{N}_{W_\varphi}(x)$. Then Proposition~\ref {prop:basic-Wulff}\eqref{char-normal-cone} and the Fenchel-Young identity imply $x \in \partial \varphi^{**}(\nu(x)) =  \partial \varphi(\nu(x))$. For smooth $\varphi$ this shows $\nabla \varphi(\nu(x)) = x$ and, hence,
\begin{align*}
\nabla \varphi(\nu(x)) \cdot (-e_d) = -x_d = \lambda\,.
\end{align*}
\end{enumerate}
\end{remark}

Furthermore, we are able to establish the stability of the Winterbottom shape improving Theorem~\ref{thm:main}. This is formulated in our second main theorem.

\begin{theorem}\label{thm:main-quant} Let $ \lambda \in \left(- \varphi(e_d), \varphi(-e_d)\right)$ and $v >0$. There is a constant $C > 0$ such that for every set $E \subseteq H^+$ of finite perimeter with $|E|=v$ there exists $\tau \in H$ that verifies 
\begin{align*}
\frac{| E \triangle (W_{\lambda,\varphi}(v)+ \tau) |^2}{v^2} \le C \left(\frac{ F_{\lambda,\varphi}(E) - F_{\lambda,\varphi}(W_{\lambda,\varphi}(v))}{F_{\lambda,\varphi}(W_{\lambda,\varphi}(v))} \right)\,.
\end{align*}
\end{theorem}

\section{Proof of the main theorems} \label{sec:proofs}

In this section we prove Theorem~\ref{thm:main} and Theorem~\ref{thm:main-quant}. In the following we will make use of various unconstrained auxiliary functionals. If $\psi : \R^n \to \R$ is a homogeneous function of degree one (not necessarily convex or bounded from below), for a set $E \subseteq H^+$ of finite perimeter we define
\begin{align*}
P_{\psi}(E) = \int_{\partial^*E} \psi(\nu)\,\,\mathrm{d}\mathcal{H}^{d-1}\,.
\end{align*}
Now suppose $\varphi_\lambda \colon \mathbb{R}^d \to \mathbb{R}$ is given by
\begin{align*}
\varphi_\lambda(\nu) =\begin{cases}  \lambda t  &\text{if }  \nu=-t e_d \text{ with } t > 0\,,\\   
\varphi(\nu) &\text{otherwise.}
\end{cases}
\end{align*}

\begin{lemma}\label{lem:auxiliary} Let $- \varphi(e_d) < \lambda < \varphi(-e_d)$. Then,
\begin{align*}
P_{\varphi_\lambda}(E) \leq F_{\lambda,\varphi}(E) 
\end{align*}
for all sets of finite perimeter $E \subseteq H^+$ with equality if and only if
\begin{align*}
\mathcal{H}^{d-1}\left(\{x \in \partial^* E \colon \nu_E(x) =-e_d\}\setminus H\right) =0\,.
\end{align*}
In particular, $W_{\lambda,\varphi}(1) \subsetneq W_\varphi(1)$ and 
\begin{align*}
P_{\varphi_\lambda}(W_{\lambda,\varphi}(1)) = F_{\lambda,\varphi}(W_{\lambda,\varphi}(1))\,. 
\end{align*}
\end{lemma}

\begin{proof} The first two statements are direct from the definition of $ P_{\varphi_\lambda}$. As $W_{\lambda,\varphi}(1)$ is convex and $W_{\lambda,\varphi}(1) \subseteq \{x_d\geq 0\}$, it suffices to show $ W_{\lambda,\varphi}(1) \cap H \neq \emptyset$, or equivalently $W_\varphi \cap \{x_d=-\lambda\}\neq \emptyset$. In fact, if there exists $x \in \partial W_{\lambda,\varphi}(1)$ such that $x_d >0$ and $-e_d=\nu(x) \in \mathcal{N}_{W_{\lambda,\varphi}(1)}(x)$, then
\begin{align*}
-x_d = x \cdot(-e_d) \geq y\cdot (-e_d) = -y_d \quad \forall y \in W_{\lambda,\varphi}(1) \iff y_d \geq x_d >0 \quad \forall y \in W_{\lambda,\varphi}(1)\,.
\end{align*} 
This contradicts $W_{\lambda,\varphi}(1) \cap H \neq \emptyset$.  In order to see that $W_\varphi \cap \{x_d=-\lambda\}\neq \emptyset$ we choose $x_+, x_- \in W_\varphi$ such that $\pm e_d \in \mathcal{N}_{W_\varphi}(x_\pm)$. Then, by Proposition~\ref{prop:basic-Wulff}\eqref{eq-normal-cone} we have $(x_\pm)_d = \pm x_\pm \cdot (\pm e_d) = \pm \varphi(\pm e_d)$. Thus by convexity of $W_\varphi$ there exists $x \in W_\varphi$ such that $x_d = - \lambda$.
\end{proof}

\begin{lemma}\label{lem:optimality} Let $0 < \lambda < \varphi(-e_d)$. Then
\begin{align*}
P_{\varphi_\lambda^{**}}(W_{\lambda,\varphi}(1)) = P_{\varphi_\lambda}(W_{\lambda,\varphi}(1))\leq P_{\varphi_\lambda}(E)
\end{align*}
for all sets of  finite perimeter $E \subseteq H^+$ such that $|E| =1$. Moreover, equality holds if and only if there exists $\tau \in H^+ \cup H$ such that $|E \Delta (W_{\lambda,\varphi}(1) +\tau)|=0$.
\end{lemma}

\begin{proof}[Proof of Lemma \ref{lem:optimality}]
 First note that the argument in the preceding proof shows that for $x \in \partial^* W_{\lambda,\varphi}$
\begin{align}\label{eq:WB-normals}
 \nu_{W_{\lambda,\varphi}}(x) = 
 \begin{cases}
   \nu_{W_{\varphi}}(x)\ne -e_d & \mbox{if } x_d > - \lambda\,,\\ 
   - e_d & \mbox{if } x_d = - \lambda\,. 
 \end{cases}
\end{align}
\EEE Since $0 < \lambda < \varphi(-e_d)$ we have that $\varphi_\lambda$ is bounded from below and that $W_{\lambda,\varphi} = W_{\varphi_\lambda}$ 
is the Wulff set of $\varphi_\lambda$.  By Proposition~\ref{prop:basic-Wulff}\eqref{W-conv-env} we therefore have 
\begin{align}\label{eq:relaxed-i}
  W_{\lambda,\varphi} = W_{\varphi_\lambda^{**}}\,.                                                                                                             
\end{align}
 Let $x \in \partial^* W_{\lambda,\varphi} = \partial^* W_{\varphi_\lambda^{**}}$. If $x_d > -\lambda$, we apply Proposition~\ref{prop:basic-Wulff}\eqref{eq-normal-cone} to both $\varphi_\lambda^{**}$ and $\varphi$ and get 
$$ \varphi_\lambda^{**}(\nu_{W_{\lambda,\varphi}}(x)) 
   = x \cdot \nu_{W_{\lambda,\varphi}}(x)
   = \varphi(\nu_{W_{\lambda,\varphi}}(x)) 
   = \varphi_\lambda(\nu_{W_{\lambda,\varphi}}(x))\,, $$ 
by \eqref{eq:WB-normals}. If $x_d = - \lambda$, we apply Proposition~\ref{prop:basic-Wulff}\eqref{eq-normal-cone} to $\varphi_\lambda^{**}$ and use \eqref{eq:WB-normals} to obtain 
$$ \varphi_\lambda^{**}(\nu_{W_{\lambda,\varphi}}(x)) 
   = x \cdot (-e_d) 
   = \lambda\, 
   = \varphi_\lambda(\nu_{W_{\lambda,\varphi}}(x)), $$ 
too. \EEE It follows that 
\begin{align}\label{eq:relaxed-ii}
P_{\varphi_\lambda^{**}}(W_{\lambda,\varphi}) = P_{\varphi_\lambda}(W_{\lambda,\varphi})\,. 
\end{align}
We can now conclude by referring to known results \cite{Taylor,Fonseca,FonsecaMueller} on the unconstrained functional $P_{\varphi_\lambda^{**}}$ as, by  \eqref{eq:relaxed-ii} and \eqref{eq:relaxed-i}, for any set $E \subseteq H^+$ of finite perimeter with $|E|=1$ we have 
\begin{align*}
  P_{\varphi_\lambda}(W_{\lambda,\varphi}(1)) 
  = P_{\varphi^{**}_\lambda}(W_{\lambda,\varphi}(1)) 
  = P_{\varphi^{**}_\lambda}(|W_{\varphi_\lambda^{**}}|^{-1/d}W_{\varphi_\lambda^{**}}) 
  \leq P_{\varphi^{**}_\lambda}(E) \leq P_{\varphi_\lambda}(E)
\end{align*}
with equality only if $E = W_{\lambda,\varphi}(1) +\hat{\tau}$ such that $E \subseteq H^+$, i.e. $E= W_{\lambda,\varphi}(1) + \tau$ with $\tau \in H^+ \cup H$.
\end{proof}

The following change of coordinates allows to reduce to the case $\lambda > 0$. 
Given  $x_0 \in \mathrm{int}\, (W_\varphi)$ we define
\begin{align*}
\varphi_{x_0}(\nu) = \varphi(\nu) - x_0\cdot\nu\,.
\end{align*}

\begin{lemma} \label{lem:lambda positive} Let $- \varphi(e_d) < \lambda < \varphi(-e_d)$. For all $x_0 \in \mathrm{int}(W_\varphi)$ there exists $\varepsilon = \varepsilon(x_0)>0$ such that  
\begin{align}\label{eqlem:epscoercivity}
\varphi_{x_0}(\nu) \geq \varepsilon\|\nu\| \quad \text{for all } \nu \in \mathbb{R}^d\,.
\end{align}
Furthermore, there exists $x_0 \in \mathrm{int}(W_\varphi)$ such that for $\lambda' = \lambda +x_0 \cdot e_d$ we have that 
\begin{align}\label{eqlem:equality}
{\rm (i)}~ 0<\lambda' <\varphi_{x_0}(-e_d)\, \quad \text{and} \quad{\rm (ii)}~ P_{(\varphi_{x_0})_{\lambda'}}(E) = P_{\varphi_\lambda}(E)\,.
\end{align}
In addition, we have 
\begin{align}\label{equiv:wulf}
{\rm (i)}~x \in W_\varphi \iff x-x_0 \in W_{{\varphi}_{x_0}} \quad \text{and} \quad{\rm (ii)}~ W_{\lambda',\varphi_{x_0}}(1) =W_{\lambda,\varphi}(1) + \tau
\end{align}
for $\tau =  |W_{\lambda,\varphi}|^{-1/d}(-x_0 +(x_0 \cdot e_d) e_d) \in H$.
\end{lemma}

\begin{proof} We first prove \eqref{eqlem:epscoercivity}.  As $\varphi_{x_0}$ is positively $1$-homogeneous, it suffices to prove the claim for $\nu \in \mathbb{S}^{d-1}$. By construction of $W_\varphi$ for all $x \in W_\varphi$ we have $x \cdot \nu \le \varphi(\nu)$ and hence \EEE
\begin{align*}
(x-x_0) \cdot \nu \leq \varphi(\nu) -x_0 \cdot\nu = \varphi_{x_0}(\nu) \quad \text{for all } \nu \in \mathbb{S}^{d-1}\,.
\end{align*}
Since $x_0 \in \mathrm{int}\, (W_\varphi)$ there is $\varepsilon >0$ such that $\overline{B}_\varepsilon(x_0) \subseteq W_\varphi $ and thus choosing $x_{\varepsilon,
\nu}=x_0+\varepsilon \nu \in \overline{B}_\varepsilon(x_0)$ we find   
\begin{align*}
\varepsilon = (x_{\varepsilon,\nu}-x_0) \cdot \nu \leq\varphi_{x_0}(\nu) \quad \text{for all } \nu \in \mathbb{S}^{d-1}\,.
\end{align*}
This shows \eqref{eqlem:epscoercivity}. 
Next, we prove \eqref{eqlem:equality}(i). To see this, observe
\begin{align*}
0 < \lambda' <\varphi_{x_0}(-e_d) \iff -x_0 \cdot e_d < \lambda < \varphi(-e_d)\,.
\end{align*}
By assumption, $\lambda < \varphi(-e_d)$. By choosing an element $\bar{x}$ with maximal last component in the compact set $W_\varphi$ and observing that $e_d \in \mathcal{N}_{W_\varphi}(\bar{x})$, with the help of Proposition~\ref{prop:basic-Wulff}\eqref{eq-normal-cone} we also get 
\begin{align*}
  -\max \{ x_0 \cdot e_d \colon x_0 \in W_\varphi\} 
  = - \bar{x} \cdot e_d
  = -\varphi(e_d)\,.
\end{align*}
Since $\lambda > - \varphi(e_d)$, there exists $x_0 \in \mathrm{int}\, (W_\varphi)$ such that $-x_0 \cdot e_d < \lambda$, too, as claimed. Next, we show \eqref{eqlem:equality}(ii). To this end, note that $(\varphi_{x_0})_{\lambda'}(\nu) = \varphi_\lambda(\nu) - x_0 \cdot\nu$ for all $\nu \in \mathbb{S}^{d-1}$ since 
$$
  (\varphi_{x_0})_{\lambda'}(\nu) 
  = \left\{ 
    \begin{aligned}
       \varphi_{x_0}(\nu) 
      &&=&&  \varphi(\nu) - x_0\cdot\nu 
      &&=&& \varphi_\lambda(\nu) - x_0\cdot\nu
      &&~\text{ if } \nu \neq -e_d\,,\\
       \lambda' 
      &&=&& \lambda - x_0 \cdot (-e_d) 
      &&=&& \varphi_\lambda(\nu) - x_0\cdot\nu
      &&~\text{ if } \nu = -e_d\,. 
   \end{aligned}
\right.
$$
Now, by the Gauss-Green Theorem for sets of finite perimeter, for $E$ a set of finite perimeter we obtain
\begin{align*}
  P_{(\varphi_{x_0})_{\lambda'}}(E) &= \int_{\partial^* E} (\varphi_{x_0})_{\lambda'}(\nu)\,\mathrm{d}\mathcal{H}^{d-1}
  = \int_{\partial^* E}\varphi_\lambda(\nu) - x_0\cdot\nu \,\mathrm{d}\mathcal{H}^{d-1} \\
  &= \int_{\partial^* E} \varphi_\lambda(\nu)\,\mathrm{d}\mathcal{H}^{d-1} - \int_E \mathrm{div}(x_0)\,\mathrm{d}x 
  = \int_{\partial^* E} \varphi_\lambda(\nu)\,\mathrm{d}\mathcal{H}^{d-1}= P_{\varphi_\lambda}(E) \,,
\end{align*}
where we used that $\mathrm{div}(x_0) =0$. To see \eqref{equiv:wulf}(i), it suffices to note that 
\begin{align*} 
  x \in W_\varphi 
  &\iff x \cdot \nu \leq \varphi(\nu) \quad \forall\, \nu \in \mathbb{S}^{d-1}
  \iff (x-x_0)\cdot \nu \leq \varphi(\nu) -x_0\cdot\nu\quad  \forall\, \nu \in \mathbb{S}^{d-1} \\
  &\iff (x-x_0) \cdot\nu \leq \varphi_{x_0}(\nu) \quad \forall\, \nu \in \mathbb{S}^{d-1} 
  \iff x-x_0 \in W_{\varphi_{x_0}}\,. 
\end{align*}
In order to prove \eqref{equiv:wulf}(ii),  Note that 
\begin{align*}
  W_{\lambda',\varphi_{x_0}} 
  &= \{x \in W_{\varphi_{x_0}} \colon x_d \geq -\lambda'\} 
  = \{x \in W_{\varphi_{x_0}} \colon x_d \geq -\lambda-x_0\cdot e_d\} \\ 
  &= \{x \in W_\varphi - x_0 \colon (x+x_0)_d \geq  -\lambda\} 
  = \{x \in W_\varphi \colon x_d \geq - \lambda\} -x_0 = W_{\lambda,\varphi}-x_0\,.
\end{align*}
 Therefore, 
\begin{align*}
  W_{\lambda',\varphi_{x_0}}(1) 
  &= |W_{\lambda',\varphi_{x_0}}|^{-1/d} (W_{\lambda',\varphi_{x_0}} +\lambda' e_d) 
  = |W_{\lambda,\varphi}|^{-1/d} (W_{\lambda,\varphi} - x_0 +(\lambda +x_0 \cdot e_d) e_d) \\
  &= |W_{\lambda,\varphi}|^{-1/d} (W_{\lambda,\varphi} + \lambda \, e_d - x_0 +(x_0 \cdot e_d) e_d) 
  = W_{\lambda,\varphi}(1) + \tau\,, 
\end{align*}
where $\tau =  |W_{\lambda,\varphi}|^{-1/d}(-x_0 +(x_0 \cdot e_d) e_d) \in H$. 
\end{proof}
We are now in position to prove the main Theorems.

\begin{proof}[Proof of Theorem \ref{thm:main}] By Remark \ref{rem:scaling} it suffices to prove the theorem only for $v=1$. To this end, let $E \subseteq H^+$ be such that $|E|=1$. We choose $x_0 \in\mathrm{int}\, (W_{\varphi})$ and $\lambda' = \lambda +x_0 \cdot e_d$ such that the assertions of Lemma \ref{lem:lambda positive} hold true. Using Lemma~\ref{lem:auxiliary},\eqref{eqlem:equality}(ii), and \eqref{equiv:wulf}(ii) we obtain 
\begin{align*}
  F_{\lambda,\varphi}(W_{\lambda,\varphi}(1))
  &= P_{\varphi_\lambda}(W_{\lambda,\varphi}(1)) 
  = P_{\varphi_\lambda}(W_{\lambda',\varphi_{x_0}}(1)) 
  =  P_{(\varphi_{x_0})_{\lambda'}}(W_{\lambda',\varphi_{x_0}}(1))
\end{align*}
and, for any set $E$ of finite perimeter,
\begin{align*}
  F_{\lambda,\varphi}(E)
  &\ge P_{\varphi_\lambda}(E)
  = P_{(\varphi_{x_0})_{\lambda'}}(E) \,.
\end{align*}
Finally applying Lemma \ref{lem:optimality} to $\varphi_{x_0}$ and the positive $\lambda'$ gives $P_{(\varphi_{x_0})_{\lambda'}}(W_{\lambda',\varphi_{x_0}}(1)) \le P_{(\varphi_{x_0})_{\lambda'}}(E)$ if $|E|=1$, and we have shown 
\begin{align*}
  F_{\lambda,\varphi}(W_{\lambda,\varphi}(1))
  \le F_{\lambda,\varphi}(E)
\end{align*}
for such $E$. In case of equality, we also have $P_{(\varphi_{x_0})_{\lambda'}}(W_{\lambda',\varphi_{x_0}}(1)) = P_{(\varphi_{x_0})_{\lambda'}}(E)$ and Lemma \ref{lem:optimality} implies $E = W_{\lambda',\varphi_{x_0}}(1) + \tau' = W_{\lambda,\varphi}(1) + \tau + \tau'$ for $\tau$ as above and some $\tau' \in H^+ \cup H$. As also $ P_{\varphi_\lambda}(E) = F_{\lambda,\varphi}(E)$, necessarily $\tau + \tau' \in H$. This concludes the proof. 
\end{proof}

We are now in position to prove Theorem \ref{thm:main-quant} establishing the stability of the Winterbottom shape.

\begin{proof}[Proof of Theorem \ref{thm:main-quant}] Let $x_0 \in \mathrm{int}(W_\varphi)$ and $\lambda'=\lambda+x_0\cdot e_d$ be such that the assertions of  Lemma~\ref{lem:lambda positive} hold true. By properly rescaling we may without loss of generality assume $|E|=1$. In the following, we set $W = W_{\lambda,\varphi}(1)$ and $W'=W_{\lambda',\varphi_{x_0}}(1)$. We first want to prove that there exists $\tau \in \mathbb{R}^d$ such that
\begin{align}\label{eq:quant-phi-rel}
 | E \triangle (W + \tau) |^2 \le C ( F_{\lambda,\varphi}(E) - F_{\lambda,\varphi}(W))\,. 
\end{align}
 The quantitative isoperimetric inequality \cite[Theorem~1.1]{FigalliMaggiPratelli} for the unconstrained functional $P_{(\varphi_{x_0})_{\lambda'}^{**}}$ yields $C > 0$ such that for a suitable $\tau' \in \R^d$ we have 
\begin{align*}
|E\Delta(W'+\tau')|^2 \leq C\left(P_{(\varphi_{x_0})_{\lambda'}^{**}}(E)-P_{(\varphi_{x_0})_{\lambda'}^{**}}(W')\right) \,.
\end{align*}
 Recalling that, by \eqref{equiv:wulf}(ii), there holds $W=W'+\hat{\tau}$ for some $\hat{\tau} \in H$ and setting $\tau = \tau'-\hat{\tau}$, we obtain
\begin{align}\label{ineq:intermediatequant-phi-rel}
|E\Delta(W+\tau)|^2=|E\Delta(W'+\tau')|^2\leq C\left(P_{(\varphi_{x_0})_{\lambda'}^{**}}(E)-P_{(\varphi_{x_0})_{\lambda'}^{**}}(W')\right) \,.
\end{align}
 Now,  by Lemma~\ref{lem:auxiliary},  Lemma~\ref{lem:lambda positive}, and, as $\lambda'>0$, by Lemma~\ref{lem:optimality} we have
\begin{align*}
P_{(\varphi_{x_0})_{\lambda'}^{**}}(E) \leq P_{(\varphi_{x_0})_{\lambda'}}(E)=P_{\varphi_{\lambda}}(E) \le F_{\lambda,\varphi}(E)
\end{align*} 
and   
\begin{align*}
P_{(\varphi_{x_0})_{\lambda'}^{**}}(W') = P_{(\varphi_{x_0})_{\lambda'}}(W')=P_{\varphi_{\lambda}}(W')=P_{\varphi_{\lambda}}(W) =F_{\lambda,\varphi}(W)\,,
\end{align*}
where we again used that $W= W'+\hat{\tau}$ with $\hat{\tau} \in H$. Therefore
\begin{align*}
P_{(\varphi_{x_0})_{\lambda'}^{**}}(E)-P_{(\varphi_{x_0})_{\lambda'}^{**}}(W') \leq  F_{\lambda,\varphi}(E)-F_{\lambda,\varphi}(W)\,.
\end{align*}
 Together with \eqref{ineq:intermediatequant-phi-rel}, this yields \eqref{eq:quant-phi-rel}.

The challenge that remains is to show that $\tau$ can be chosen in $H$, i.e., such that $\tau_d = 0$. In order to do that we distinguish cases. 
\smallskip

\noindent{\em Case 1.} Assume \eqref{eq:quant-phi-rel} holds true with $\tau_d < 0$. By convexity of $W$ we have 
\begin{align*}
  |E \triangle (W + \tau)| 
  \ge |\{x \in \tau + W \colon x_d < 0\}| 
  = |\{x \in W \colon x_d < - \tau_d\}| \ge  c \min\{ -\tau_d ,1\} 
\end{align*}
for a constant $c > 0$. Since also 
\begin{align*}
  |(W + \tau) \triangle (W + \tau - \tau_d e_d)| 
  \le  C \min\{ -\tau_d ,1\} \, 
\end{align*}
(note that trivially $|(W + \tau) \triangle (W + \tau - \tau_d e_d)| \le 2|W| \le C$), we find that 
\begin{align*}
  |E \triangle (W + \tau - \tau_d e_d)| 
  \le |E \triangle (W + \tau)| + |(W + \tau) \triangle (W + \tau - \tau_d e_d)| 
  = (1 + C c^{-1}) |E \triangle (W + \tau)|\,. 
\end{align*}
This implies the claim in view of \eqref{eq:quant-phi-rel}. 
\smallskip 

\noindent {\em Case 2.} Assume \eqref{eq:quant-phi-rel} holds true with $\tau_d > 0$. In the following we denote by 
$$ W_\varphi(v) 
   = \left(\frac{v}{|W_{\varphi}|}\right)^{\frac{1}{d}} W_{\varphi} $$ 
the rescaled Wulff set of $\varphi$ with volume $v > 0$. Since in the partial drying/wetting regime $W_\varphi(1)$ (placed in $H^+$) is not optimal for $F_{\lambda, \varphi}$, we have $P_\varphi(W_{\varphi}(1)) > F_{\lambda,\varphi}(W_{\lambda,\varphi}(1))$ and so we can define a positive constant $c_0$ by asking that 
\begin{align*} 
  \left( 1 - c_0 \right)^{\frac{d-1}{d}} P_\varphi(W_\varphi(1)) - (\varphi(e_d) + \varphi(-e_d)) c_0 
  = F_{\lambda,\varphi}(W_{\lambda,\varphi}(1)) + \delta\,, 
\end{align*}
where $\delta := \frac{1}{2} (P_\varphi(W_\varphi(1)) - F_{\lambda,\varphi}(W_{\lambda,\varphi}(1)) > 0$.
\smallskip 

\noindent {\em Case 2a.} Suppose that $|\{ x \in E \colon x_d < \tau_d \}| \ge c_0 \min\{\tau_d,1\}$. Then
\begin{align*}
  |E \triangle (W + \tau)| 
  \ge c_0 \min\{\tau_d,1\}  
\end{align*}
and, similarly as above, from 
\begin{align*}
  |(W + \tau) \triangle (W + \tau - \tau_d e_d)| 
  \le C \min\{\tau_d,1\}, 
\end{align*}
we conclude 
\begin{align*}
  |E \triangle (W + \tau - \tau_d e_d)| 
  \le |E \triangle (W + \tau)| + |(W + \tau) \triangle (W + \tau - \tau_d e_d)| 
  = (1 + C c_0^{-1}) |E \triangle (W + \tau)|\,. 
\end{align*}
Again this finishes the proof with the help of 
\eqref{eq:quant-phi-rel}. 
\smallskip 

\noindent {\em Case 2b.} Assume that $|\{ x \in E \colon x_d < \tau_d \}| < c_0 \min\{\tau_d,1\}$. 
Then there is an $\eps \in (0, \tau_d)$ such that 
$$ \mathcal{H}^{d-1}( E \cap (\eps e_d + H) ) \le c_0\min\{1, \tau_d^{-1}\}. $$ 
We cut along this hyperplane and set $\tilde{E} = \{x \in E : x_d > \eps \}$. With this set we then get  
$$ F_{\lambda,\varphi}(E) 
   \ge F_{\lambda,\varphi}(E \setminus \tilde{E}) + P_\varphi(\tilde{E}) - (\varphi(e_d) + \varphi(-e_d)) \mathcal{H}^{d-1}(E \cap (\eps e_d + H)). $$ 
By Remark~\ref{rmk:drying-wetting}\eqref{positive-m}, the first term on the right hand side is non-negative and we get the bound 
\begin{align*}
  F_{\lambda,\varphi}(E) 
  &\ge P_\varphi(\tilde{E}) - (\varphi(e_d) + \varphi(-e_d)) c_0\min\{1,\tau_d^{-1}\} \\ 
  &\ge P_\varphi(W_\varphi(|\tilde{E}|)) - (\varphi(e_d) + \varphi(-e_d)) c_0 \\ 
  &= |\tilde{E}|^{\frac{d-1}{d}} P_\varphi(W_\varphi(1)) - (\varphi(e_d) + \varphi(-e_d)) c_0\,,  
\end{align*}
where in the second step we have used that the Wulff shape $W_\varphi(|\tilde{E}|)$ minimizes the unconstrained functional $P_\varphi$ among sets of volume $|\tilde{E}|$. As by assumption $|\tilde{E}|\ge 1 - c_0$, it follows that 
\begin{align}\label{ineq:energyex}
  F_{\lambda,\varphi}(E) 
  &\ge  \left( 1 - c_0 \right)^{\frac{d-1}{d}} P_\varphi(W_\varphi(1)) - (\varphi(e_d) + \varphi(-e_d)) c_0 
  \ge F_{\lambda,\varphi}(W_{\lambda,\varphi}(1)) + \delta
\end{align}  
by our choice of $c_0$. Thus, for any $C \ge 4\delta^{-1}$ we have even for $\tau = 0$  
\begin{equation*}
  | E \triangle W |^2 
  \le (|E| + |W|)^2 
  = 4 
  \le C ( F_{\lambda,\varphi}(E) - F_{\lambda,\varphi}(W)). \qedhere  
\end{equation*}
\end{proof}

\section*{Acknowledgements} 
The research of LK was supported by the DFG through the Emmy Noether Programme (project number 509436910).  \EEE 

\EEE

\end{document}